\documentclass[a4paper,12pt]{amsart}
\usepackage[normalem]{ulem}
\usepackage{amsfonts}
\usepackage{amsmath}
\usepackage{amssymb}
\usepackage{mathrsfs}
\usepackage{hyperref}
\usepackage{graphicx}
\usepackage{algorithm}
\usepackage{listings}
\usepackage{algorithm,algorithmic}
\usepackage{amsmath,tikz}
\usetikzlibrary{matrix}
\usepackage{colortbl}
\usepackage{esint}

\numberwithin{equation}{section}
\theoremstyle{plain}
\newtheorem{thm}{Theorem}[section]

\newtheorem{lemma}[thm]{Lemma}

\theoremstyle{definition}
\newtheorem{defi}[thm]{Definition}

\begin{document}

\title[Asymptotic analysis of the spectrum of biharmonic Steklov problem]
{Asymptotic behavior and spectral distortion for  biharmonic Steklov problems on thin domains }

\author[Bauyrzhan  Derbissaly]
{Bauyrzhan Derbissaly}
\address{
  Bauyrzhan  Derbissaly: 
  \endgraf
   Institute of Mathematics and Mathematical Modeling,
   \endgraf Department of Differential Equations,
  \endgraf
  Pushkin 125 - 050010 Almaty, Kazakhstan
  \endgraf
    {\it E-mail address} {\rm derbissaly@math.kz}
 }

\author[Pier Domenico Lamberti]
{Pier Domenico Lamberti}
\address{
  Pier Domenico Lamberti: 
  \endgraf
   Universit\`{a} degli Studi di Padova, 
   \endgraf Dipartimento di Tecnica e Gestione dei Sistemi Industriali,
  \endgraf
  Stradella S. Nicola 3 - 36100 Vicenza, Italy
  \endgraf
    {\it E-mail address} {\rm lamberti@math.unipd.it}
 }

\thanks{}


\begin{abstract}
In this paper, we investigate the asymptotic behavior of the eigenvalues and eigenfunctions of a biharmonic Steklov problem defined on a thin domain in the $n$ dimensional Euclidean space degenerating to a segment.
For $n=2$ the problem models the vibrations of a  thin elastic plate with cross section represented by the given domain and mass concentrated on a free boundary.
The problem under consideration depends on a parameter $\sigma$ that in the theory of elastic plates represents the Poisson ratio of the material. Our analysis points out a distortion
in the limiting problem depending on $\sigma$ and the
space dimension $n$.
\end{abstract}

\maketitle
\section{Introduction}
\noindent
Given a bounded, Lipschitz, connected open set (briefly, a bounded domain of class $C^{0,1}$) $\Omega $ in  $\mathbb{R}^n$ with $n\geq2$ and an open subset $\Gamma$ of its boundary $\partial \Omega$, we consider the following (mixed) generalized Steklov problem for the biharmonic operator
\begin{equation}\label{1.1}
\left\{
\begin{array}{ll}
\displaystyle
\Delta^2u=0,& \ \text{in} \ \Omega,\\
\left(1-\sigma\right)u_{\nu\nu}+\sigma\Delta u+\mu u_\nu =0,& \ \text{on} \ \Gamma,
\\
-\left(1-\sigma\right)\text{div}_{\Gamma}\left(D^2u\cdot\nu\right)_{\Gamma}-\left(\Delta u\right)_{\nu}=\lambda u,& \ \text{on} \ \Gamma,\\
u=u_{\nu}=0,& \ \text{on} \ L . 
\end{array} \right.
\end{equation}
Here, $L=\partial\Omega\setminus \Gamma$, $u$ is the Steklov eigenfunction, $\lambda$ is the Steklov eigenvalue, $ -1/(n-1)<\sigma < 1$, $\mu>0$ are fixed constants, and $\nu$ is the unit outer normal to $\partial\Omega$.
By $\text{div}_{\Gamma}F:= \text{div} F-(\nabla F\cdot \nu)\nu$ we denote the tangential divergence of a vector field
$F$, while $F_{\Gamma}:= F-(F\cdot\nu)\nu$ denotes the tangential component of $F$. As it will be clear in the sequel, 
we impose Dirichlet boundary conditions $u=u_{\nu}=0$ on $L$  in order to simplify our analysis in the specific case under consideration. 

Problem \eqref{1.1} is the natural fourth order version of the  second order Steklov problem 
\begin{equation}\label{1.1bis}
\left\{
\begin{array}{ll}
\displaystyle
\Delta u=0,& \ \text{in} \ \Omega,
\\
u_{\nu}=\lambda u,& \ \text{on} \ \Gamma,\\
u=0,& \ \text{on} \ L . 
\end{array} \right. 
\end{equation}
Problem \eqref{1.1bis} can be seen as the eigenvalue problem of a Dirichlet-to-Neumann map, it has a long history  and many applications. The classical case with $L=\emptyset$ has a  fundamental role in the study of the sloshing problem concerning the small oscillations of a fluid in a small basin and in electrical prospection in which case the Dirichlet-to-Neumann map represents  the voltage-to-current map. In linear elasticity theory, problem \eqref{1.1bis} with $n=2$ provides the fundamental modes of vibrations $u$  and the corresponding eigenfrequencies $\sqrt{\lambda}$ of a free vibrating membrane $\Omega$ with boundary fixed only at $L$ and mass concentrated in $\Gamma$. The Rayleigh quotient associated with problem \eqref{1.1bis} is 
$$\frac{\int_{\Omega}|\nabla u|^2dx}{\int_{\Gamma}u^2 d\mathcal{H}},
$$
where $d\mathcal{H}=d\mathcal{H}^1$ is the one-dimensional Hausdorff measure (arch-length). We mention in passing that  in the general case  $n\geq 2$, we shall  use  the $n-1$-dimensional Hausdorff measure $d\mathcal{H}^{n-1}$. We refer to \cite{ferlamstra}  for more details on the applications of problem \eqref{1.1bis} and to \cite{lampro} for its physical interpretation   in the framework of the theory of vibrating membranes. 

Similarly, problem \eqref{1.1} with $n=2$ arises in the study of elastic plates within the framework of the Kirchhoff-Love model. This model is commonly used to describe the bending behavior, vibrations, and deformations of thin elastic plates. Specifically, the Kirchhoff-Love plate theory assumes that the plate's thickness is small compared to its other dimensions and that normal stresses through the thickness are negligible. The model is employed to predict the deformations of thin plates under various forces and boundary conditions. See \cite{DBLP} for more details. In this context, $\Omega$  represents the cross section of a thin vibrating plate with boundary  clamped only at $L$ and with mass concentrated in $\Gamma$. The parameter  $\sigma$ is the Poisson’s ratio of the material from which the plate is made. The function $u$ is the vibrating mode and represents the deflection from the midplane. The associated potential energy is given by 
\[
\int_{\Omega}\left(1-\sigma\right)|D^2u|^2
+\sigma |\Delta u|^2 dx+\mu\int_{\Gamma}u^2_{\nu} d\mathcal{H}     
\]
and 
\begin{equation}
\label{rayintro}
\frac{
\int_{\Omega}\left(1-\sigma\right)|D^2u|^2
+\sigma |\Delta u|^2 dx+\mu\int_{\Gamma}u^2_{\nu} 
d\mathcal{H}}{\int_{\Gamma}u^2 d\mathcal{H} }   
\end{equation}
is the corresponding Rayleigh quotient. Here $D^2u$ denotes the Hessian matrix of $u$.
We refer to \cite{DBLP} for a rigorous derivation and physical interpretation of problem \eqref{1.1} for $\mu=0$.

As for the second order problem, it turns out that  the set of eigenvalues of problem \eqref{1.1} is a countable set of positive real numbers, that can be arranged into a non-decreasing  sequence that diverges to $+\infty$:
\[
0<\lambda_1\leq\ \cdot\cdot\cdot\leq \lambda_n\leq...
\]

The variational formulation of problem \eqref{1.1} (that can be obtained as the Euler-Lagrange equations associated with the minimization  problem for the Rayleigh quotient \eqref{rayintro}) has the form
\begin{equation}
\label{variationalintro}
\begin{aligned}
\int_{\Omega}\left(1-\sigma\right)D^2u:D^2\varphi
+\sigma \Delta u\Delta \varphi dx&+\mu\int_{\Gamma}u_{\nu}\varphi_{\nu} d\mathcal{H}^{n-1} =
\lambda\int_{\Gamma}u\varphi d\mathcal{H}^{n-1},     
\end{aligned}
\end{equation}
for all $\varphi\in H^2_{L}\left(\Omega\right)$, in the unknowns $u\in H^2_{L}\left(\Omega\right), \ \lambda\in \mathbb{R}$, where $D^2u:D^2\varphi $ denotes the Frobenius product of $D^2u$ and $D^2\varphi$. Note that the boundary conditions imposed on $\Gamma$ in \eqref{1.1} naturally arises by integrating by parts the volume integral  in \eqref{variationalintro}, see e.g., the `Biharmonic Green Formula' in \cite[Lemma~8.61]{arrlam}.

By $H^2_{L}(\Omega)$ we denote the closure in $H^2(\Omega)$ of the space of functions vanishing in a neighborhood of $L$ and $H^2(\Omega)$ is the standard Sobolev space $H^2(\Omega)$ of real-valued functions in $L^2(\Omega)$ with weak derivatives up to order two in $L^2(\Omega)$.

If we set $\mu=+\infty$ in problem \eqref{1.1}, then we obtain the so-called (NBS)-Neumann Biharmonic Steklov problem 
\[
\left\{
\begin{array}{ll}
\displaystyle
\Delta^2u=0,& \ \text{in} \ \Omega,\\
u_\nu=0,& \ \text{on} \ \Gamma,
\\
-\left(1-\sigma\right)\text{div}_{\Gamma}\left(D^2 u\cdot\nu\right)_{\Gamma}-\left(\Delta u\right)_{\nu}=\lambda u,& \ \text{on} \ \Gamma,\\
u=u_{\nu}=0,& \ \text{on} \ L . \\
\end{array} \right.
\]
The problem (NBS) was discussed in \cite{KJR, GL, AP} for the case $\sigma=1$. It is worth noting that problem \eqref{1.1} with $\sigma=\mu=0$ was introduced in \cite{DBLP} as a natural fourth-order extension of the classical Steklov problem for the Laplacian, and was also considered  in \cite{buosostek} for $\sigma\in (-1/(n-1),1)$ and $\mu=0$. Problem \eqref{1.1} with $\mu>0$ was introduced in \cite{PDLLP} to characterize the trace spaces of functions in $H^2(\Omega)$ when $\Omega$ is a bounded Lipschitz domain in $\mathbb{R}^n$. We refer also to \cite{buoken} for further discussion about Robin and Steklov problems for the biharmonic operator.  In the cited papers $\Gamma=\partial\Omega$.

The main ideas of this paper are inspired by two works: \cite{PDL} and \cite{BD}. In \cite{PDL}, the authors examine the spectral properties of the biharmonic operator with homogeneous Neumann boundary conditions on a planar dumbbell domain. This domain consists of two domains connected by a narrow channel. Their study focuses on the asymptotic behavior of the eigenvalues as the channel thickness tends to zero.

On the other hand, \cite{BD} investigates the asymptotic behavior of the eigenvalues and eigenfunctions of the classical second-order Steklov problem in a dumbbell domain, which consists of two  domains connected by a thin tube with vanishing width. The authors demonstrate that, in this case, all eigenvalues converge to zero, with the convergence rate governed by a power of the tube’s width, which scales with the eigenvalues of a corresponding one-dimensional problem.

The study of eigenvalue problems for differential operators on thin domains is a classical topic that attracts the attention of many authors; see, for example, \cite{JM}, \cite{VP}, \cite{FP},  \cite{CD}, \cite{FEPR}, \cite{GA}, \cite{JC}, \cite{SA}, \cite{MC}, and the references therein.

We have the following variational characterization of problem \eqref{1.1}:
\[
\lambda_k=\inf_{S_k}
\sup_{\substack{ u\in S_k\\ {\rm Tr}\, u \ne 0}}\frac{Q_{\sigma,\mu,\Omega}\left(u,u\right)}{\| u\|^2_{L^2(\Gamma)}},
\]
where
\[
Q_{\sigma,\mu,\Omega}\left(u,\varphi\right)=
Q_{\sigma,\Omega}\left(u,\varphi\right)+\mu
(u_\nu,\varphi_\nu)_{L^2(\Gamma)},    
\]
and
\[
Q_{\sigma,\Omega}\left(u,\varphi\right)=\left(1-\sigma\right)\int_{\Omega}D^2u:D^2\varphi dx
+\sigma\int_{\Omega}\Delta u\Delta \varphi dx,
\]
for all $u, \varphi\in H^2\left(\Omega\right)$. As usual,  $(\cdot,\cdot)_{L^2(\Gamma)}$ and $ \| \cdot \|_{L^2(\Gamma)}$ denote the scalar product and the corresponding norm in $L^2(\Gamma)$. Moreover, ${\rm Tr}\, u$ denotes the trace of $u$ on $\Gamma$ and the infimum is taken over all $k$-dimensional subspaces $S_k$ of the Sobolev space
$H^2_{L}\left(\Omega\right)$. 

We now describe the specific class of domains where we are going to formulate our Steklov problem.  Let $\Omega_\epsilon$ be a thin domain in $\mathbb{R}^n$ defined by
\[
\Omega_\epsilon=\left\{(x_1,x')\in\mathbb{R}^n:\ -l<x_1<l, \ |x'|<\epsilon\rho(x_1)\right\},
\]
where $l>0$ and $\rho\in C^0([-l,l])\cap C^\infty\left(-l,l\right)$ is a positive function. Here, $x'=(x_2,...,x_n)$. We note that  the choice of a symmetric interval $[-l,l]$ is in the spirit of \cite{BD} but for our purposes we may choose any interval  $[a,b]$.

We define $\Gamma_\epsilon$ as follows:
\[
\Gamma_\epsilon=\left\{(x_1,x')\in\mathbb{R}^n:\ -l\leq x_1\leq l, \ |x'|=\epsilon \rho(x_1) \right\}.
\]
We denote $L_\epsilon=L^+_\epsilon\cup L^-_\epsilon$, where
\[
\begin{aligned}
L^+_\epsilon&=\left\{(x_1,x')\in\mathbb{R}^n:\  x_1=l, \ |x'|\leq \epsilon\rho\left(l\right)\right\},
\\
L^-_\epsilon&=\left\{(x_1,x')\in\mathbb{R}^n:\  x_1=-l, \ |x'|\leq \epsilon \rho\left(-l\right)\right\}.
\end{aligned}
\]

In the thin domain $\Omega_\epsilon$ we consider the problem
\begin{equation}\label{1.2}
\left\{
\begin{array}{ll}
\displaystyle
\Delta^2u_\epsilon=0,& \ \text{in} \ \Omega_\epsilon,\\
\left(1-\sigma\right)(u_\epsilon)_{\nu_\epsilon\nu_\epsilon}+\sigma\Delta u_\epsilon=-\mu (u_\epsilon)_{\nu_\epsilon},& \ \text{on} \ \Gamma_\epsilon,
\\
-\left(1-\sigma\right)\text{div}_{\Gamma_\epsilon}\left(D^2u_\epsilon\cdot\nu_\epsilon\right)_{\Gamma_\epsilon}-\left(\Delta u_\epsilon\right)_{\nu_\epsilon}=\lambda_\epsilon u_\epsilon,& \ \text{on} \ \Gamma_\epsilon,\\
u_\epsilon=(u_\epsilon)_{\nu_\epsilon}=0,& \ \text{on} \ L_\epsilon.\\
\end{array} \right.
\end{equation}
The weak formulation of problem \eqref{1.2} takes the form
\begin{equation}\label{1.3}
Q_{\sigma,\mu,\Omega_\epsilon}\left(u_\epsilon,\varphi\right)=\lambda_\epsilon
(u_{\epsilon},\varphi)_{L^2(\Gamma_{\epsilon})}, \ \text{for all} \ \varphi\in H_{L_\epsilon}^2\left(\Omega_\epsilon\right).
\end{equation}
The main goal of this paper is to investigate the behavior of $\left(\lambda_\epsilon,u_\epsilon\right)$ as $\epsilon\rightarrow{0}$. 
 
As above, we represent  the eigenvalues of problem \eqref{1.3} by means of a non-decreasing sequence $\lambda_{\epsilon , k}$, with $k\in \mathbb{N}$ where each eigenvalue is repeated as many times as its multiplicity.  We denote by $u_{\epsilon , k}$, $k\in\mathbb{N}$, a corresponding  sequence of eigenfunctions which is assumed to be orthonormal  with respect to the scalar product of $L^2(\Gamma_{\epsilon})$, that is $\int_{\Gamma_{\epsilon}}u_{\epsilon , h}u_{\epsilon , k}d\mathcal{H}^{n-1}=\delta_{hk} $ for all $h,k\in \mathbb{N}$. It turns out that the asymptotic behavior of the Steklov eigenvalues $\lambda_{\epsilon , k}$ and eigenfunctions $u_{\epsilon , k}$ as $\epsilon \to 0$ is governed by the 4th-order Sturm-Liouville problem \eqref{1.4} defined in the limiting segment $[-l,l]$. We note in particular that the eigenfunctions $v_k$, $k\in \mathbb{N}$ of this Sturm-Liouville problem, have to be normalized in the weighted space $L^2_{n,\rho}(-l,l):=L^2\left(\left(-l, l\right); (n-1)w_{n-1}\rho^{n-2}dx_1\right)$ where $w_{n-1}$ denotes the volume of the unit ball in $\mathbb{R}^{n-1}$. 

We  are ready present the main result of this paper. Note that,  since the domain of $u_{\epsilon ,k}$
changes with $\epsilon$, here we pull $u_{\epsilon}$  back to $\Omega_1$ by considering $u_{\epsilon, k}\circ \Phi_{\epsilon}$, where
 $\Phi_\epsilon: \ \Omega_1\rightarrow{\Omega_\epsilon}$ is defined by $\Phi_\epsilon(x_1,x')=(x_1,\epsilon x')$ for all $(x_1, x')\in \Omega_1$. On the other hand, the limiting eigenfunctions $u_{k}$, that depend only on the variable $x_1$, are extended  constantly in the remaining variables to the whole of $\Omega_1$. 

\begin{thm}\label{thm 1.1} 
Let $n\geq2$. Then
\[
\lambda_{\epsilon , k}\sim \lambda_k\epsilon, \ \text{as} \ \epsilon \rightarrow{0},
\]
for all $k\in \mathbb{N}$,  where $\lambda_k$ are the eigenvalues of the 4th-order Sturm-Liouville problem
\begin{equation}\label{1.4}
\left\{
\begin{array}{llll}
\displaystyle
&\left(1-\sigma^2\mathcal{N}\right)\frac{d^2}{dx_1^2}\left(\rho^{n-1}\frac{d^2V}{dx_1^2}\right)=\lambda(n-1) \rho^{n-2} V,\ \ \ {\rm in}\ \left(-l,l\right),
\vspace{1mm}\\&
V\left(-l\right)=\frac{d V}{dx_1}\left(-l\right)=0,
\vspace{1mm}\\&
V\left(l\right)=\frac{d V}{dx_1}\left(l\right)=0,
\end{array} \right.
\end{equation}
and $\mathcal{N}=(n-1)/(1-2\sigma+\sigma n)$. Moreover, there exists an orthonormal basis  of eigenfunctions $v_k$, $k\in \mathbb{N}$ of \eqref{1.4} in $L^2_{n,\rho}(-l,l)$ such that, possibly passing to a subsequence, 
\[
\epsilon^{\frac{n-2}{2}}u_{\epsilon,k} \circ \Phi_{\epsilon}\rightharpoonup v_k \ \text{in} \ H^2(\Omega_1),
\]
as $\epsilon\to0$, where $v_k$ is constantly extended in the variables $x'$.
\end{thm}

Note the appearance of the distortion factor $1-\sigma^2\mathcal{N}$ in front of the $4$th-order operator in the limiting problem when $\sigma \ne0$. In the case $n=2$ this factor was found in \cite{PDL}.

 We also note that while we have formulated our problem in the spirit of  \cite{BD} to enable comparison, our method of proof is entirely different. In particular, here we prove the spectral convergence of the ($\epsilon$-rescaled) $n$-dimensional Steklov problem to the 4th-order Sturm-Liouville problem in the sense of Vainnikko~\cite{GMV}.

\section{Preliminaries and notation}\label{sect. 2}
\noindent
When considering the biharmonic Steklov problem in the domain $\Omega_\epsilon$, the corresponding Hilbert spaces $\mathcal{H}_\epsilon$ depend on $\epsilon$, complicating the direct application of the standard notion of compact convergence. To bypass this, we employ suitable connecting systems that facilitate the transition from the variable Hilbert spaces defined on $\Omega_\epsilon$ to the fixed limiting Hilbert space defined on $\Omega_1$. This methodology integrates various concepts and results from the works of Stummel \cite{FS} and Vainikko \cite{GMV}, which have been further elaborated in \cite{JMA, AC}. Notably, we utilize the concept of $E$-compact convergence.

In the spirit of \cite{JMA}, we denote by $\mathcal{H}_\epsilon$ a family of Hilbert spaces for $\epsilon\in [0, \epsilon_0]$ and assume the existence of a family of linear operators $E_\epsilon: \  \mathcal{H}_0 \rightarrow \mathcal{H}_\epsilon$ such that
\begin{equation}\label{2.1}
\|E_\epsilon f\|_{\mathcal{H}_\epsilon}\xrightarrow{\epsilon\rightarrow 0}\|f\|_{\mathcal{H}_0}, \ \text{for all} \ f\in\mathcal{H}_0.
\end{equation}

\begin{defi}\label{def 2.1}
We say that a family $\{f_\epsilon\}_{0<\epsilon\leq\epsilon_0}$, with $f_\epsilon\in \mathcal{H}_\epsilon, E$-converges to $f\in \mathcal{H}_0$ if $\|f_\epsilon-E_\epsilon f\|_{\mathcal{H}_\epsilon}\rightarrow0$
as $\epsilon\rightarrow0$. We write this as $f_\epsilon\xrightarrow{E}f$.
\end{defi}

\begin{defi}\label{def 2.2}
Let $\{B_\epsilon\in\mathcal{L}\left(\mathcal{H}_\epsilon\right): \ \epsilon\in(0,\epsilon_0]\}$ be a family of linear and continuous operators. We say
that $\{B_\epsilon\}_{0<\epsilon\leq\epsilon_0}$ $E$-converges to $B_0\in\mathcal{L}\left(\mathcal{H}_0\right)$ as $\epsilon\rightarrow0$ if $B_\epsilon f_\epsilon\xrightarrow{E}B_0f$ whenever $f_\epsilon\xrightarrow{E}f$. We write this as
$B_\epsilon\xrightarrow{EE}B_0$.
\end{defi}

\begin{defi}\label{def 2.3}
Let $\{f_\epsilon\}_{0<\epsilon\leq\epsilon_0}$ be a family such that $f_\epsilon\in \mathcal{H}_\epsilon$. We say that $\{f_\epsilon\}_{0<\epsilon\leq\epsilon_0}$ is precompact if
for any sequence $\epsilon_n\rightarrow0$ there exist a subsequence $\{\epsilon_{n_k}\}_{k\in\mathbb{N}}$ and $f\in \mathcal{H}_0$ such that $f_{\epsilon_{n_k}}\xrightarrow{E}f$ as $k\to\infty$.
\end{defi}

\begin{defi}\label{def 2.4}
We say that $\{B_\epsilon\}_{0<\epsilon\leq\epsilon_0}$ with $B_\epsilon\in\mathcal{L}\left(\mathcal{H}_\epsilon\right)$ and $B_\epsilon$ compact, converges compactly to a
compact operator $B_0\in\mathcal{L}\left(\mathcal{H}_0\right)$ if $B_\epsilon\xrightarrow{EE}B_0$ and for any family $\{f_\epsilon\}_{0<\epsilon\leq\epsilon_0}$ such that $f_\epsilon\in \mathcal{H}_\epsilon$, $\|f_\epsilon\|_{\mathcal{H}_\epsilon}=1$,
we have that $\{B_\epsilon f_\epsilon\}_{0<\epsilon\leq\epsilon_0}$ is precompact in the sense of Definition \ref{def 2.3}. We write this as $B_\epsilon\xrightarrow{C}B_0$.
\end{defi}

The $E$-compact convergence implies spectral stability. Namely, we have the following result where by `generalized eigenfunction' associated to $m$ eigenvalues  we mean a linear combination of $m$ eigenfunctions associated to those egenvalues.

\begin{thm}\label{th 2.5}
Let $\{B_\epsilon\}_{0\le \epsilon\leq\epsilon_0}$ be a family of non-negative, compact self-adjoint operators in the
Hilbert spaces $\mathcal{H}_\epsilon$. Assume that their eigenvalues are given by $\{\lambda_k(\epsilon)\}^{\infty}_{k=1}$. If $B_\epsilon\xrightarrow{C}B_0$, then there is spectral convergence of $B_\epsilon$ to $B_0$ as $\epsilon\rightarrow0$.
In particular, the following statements hold: 
\begin{itemize}
 \item[(i)] For every $k\in  \mathbb{N}$ we have $\lambda_k(\epsilon )\to \lambda_k(0)$ as $\epsilon \to 0$.
\item[(ii)] If $u_k(\epsilon)$, $k\in \mathbb{N}$, is an orthonormal sequence of eigenfunctions associated with the eigenvalues  $\lambda_k(\epsilon )$ then
 there exists an orthonormal sequence of eigenfunctions  $u_k(0)$, $k\in  \mathbb{N}$ associated with  $\lambda_k(0 )$, $k\in  \mathbb{N}$ such that, possibly passing to a subsequence, $u_k(\epsilon )\xrightarrow{E} u_k(0)$.
 \item[(iii)]  Given  $m$ eigenvalues $\lambda_k(0),  \dots , \lambda_{k+m-1}(0)$ with
$\lambda_k(0)\ne \lambda_{k-1}(0)$ and $\lambda_{k+m-1}(0)$ $\ne \lambda_{k+m}(0)$
 and corresponding orthonormal eigenfunctions $u_k(0),\dots,u_{k+m-1}(0)$
 there exist $m$ orthonormal   generalized eigenfunctions   $v_k(\epsilon ),  \dots , v_{k+m-1}(\epsilon )$  associated with  $\lambda_k(\epsilon ),  \dots ,  
  \lambda_{k+m-1}(\epsilon )$   such that $v_{k+j}(\epsilon )\xrightarrow{E} u_{k+j}(0)$  for all $j=0, 1,\dots , m-1$.
 \end{itemize}    
\end{thm}

We refer to \cite[Theorem~2.5]{AFPDL}, \cite[Theorem~6.3]{GMV}, \cite[Theorem~1]{SS}, see also \cite[Theorem~4.10]{JMA}, \cite[Theorem~5.1]{EA} and
\cite[Theorem~3.3]{AC} 
for more details on spectral convergence. 

In this paper, we apply Theorem \ref{th 2.5} to the resolvent operators associated with the given problems. Before doing this, we find it convenient to add a penalty term in the equation under consideration. Namely,  we set  $\overline{\lambda}_\epsilon:=\lambda_\epsilon+\epsilon$ and we consider the problem 
\begin{equation}\label{2.2}
\left\{
\begin{array}{ll}
\displaystyle
\Delta^2u_\epsilon=0,& \ \text{in} \ \Omega_\epsilon,\\
\left(1-\sigma\right)(u_\epsilon)_{\nu_\epsilon\nu_\epsilon}+\sigma\Delta u_\epsilon=-\mu (u_\epsilon)_{\nu_\epsilon},& \ \text{on} \ \Gamma_\epsilon,
\\
-\left(1-\sigma\right)\text{div}_{\Gamma_\epsilon}\left(D^2u_\epsilon\cdot\nu_\epsilon\right)_{\Gamma_\epsilon}-\left(\Delta u_\epsilon\right)_{\nu_\epsilon}+\epsilon u_\epsilon=\overline{\lambda}_\epsilon u_\epsilon,& \ \text{on} \ \Gamma_\epsilon,\\
u_\epsilon=(u_\epsilon)_{\nu_\epsilon}=0,& \ \text{on} \ L_\epsilon.\\
\end{array} \right.
\end{equation}
Note that the variational formulation of problem \eqref{2.2} is immediately deduced by   \eqref{1.3} and is expressed by
\begin{equation}\label{2.3}
Q_{\sigma,\mu,\Omega_\epsilon}\left(u_\epsilon,\varphi\right)+\epsilon (u_\epsilon,\varphi)_{L^2(\Gamma_{\epsilon})} =\overline{\lambda}_\epsilon
(u_\epsilon,\varphi)_{L^2(\Gamma_{\epsilon})}
   , \ \text{for all} \ \varphi\in H_{L_\epsilon}^2\left(\Omega_\epsilon\right).
\end{equation}
Clearly, studying the asymptotic behavior of $\lambda_{\epsilon}$ is equivalent to studying the behavior of  $\overline{\lambda}_\epsilon$ as $\epsilon\rightarrow{0}$.
To do so, we plan to apply Theorem~\ref{th 2.5} to the resolvent operator associated with problem \eqref{2.2}.  Namely,  we consider the following problem with the datum $f_\epsilon\in L^2(\Gamma_\epsilon)$
\begin{equation}\label{2.4}
\left\{
\begin{array}{ll}
\displaystyle
\Delta^2u_{\epsilon}=0,& \ \text{in} \ \Omega_\epsilon,\\
\left(1-\sigma\right)(u_{\epsilon})_{\nu_\epsilon\nu_\epsilon}+\sigma\Delta u_{\epsilon}=-\mu (u_{\epsilon})_{\nu_\epsilon},& \ \text{on} \ \Gamma_\epsilon,
\\
-\left(1-\sigma\right)\text{div}_{\Gamma_\epsilon}\left(D^2u_{\epsilon}\cdot\nu_\epsilon\right)_{\Gamma_\epsilon}-\left(\Delta u_{\epsilon}\right)_{\nu_\epsilon}+\epsilon u_{\epsilon}=f_{\epsilon},& \ \text{on} \ \Gamma_\epsilon.\\
u_{\epsilon}=(u_{\epsilon})_{\nu_\epsilon}=0,& \ \text{on} \ L_\epsilon,
\end{array} \right.
\end{equation}
and we study the behavior of the solution $u_{\epsilon}$ upon variation of $f_{\epsilon}$. 

Note that the weak formulation of this problem is 
\begin{equation}\label{2.5}
Q_{\sigma,\mu,\Omega_\epsilon}\left(u_\epsilon,\phi\right)+\epsilon
(u_\epsilon,\phi)_{L^2(\Gamma_{\epsilon})}
=
(f_\epsilon,\phi)_{L^2(\Gamma_{\epsilon})}, \ \text{for all} \ \phi\in H_{L_\epsilon}^2\left(\Omega_\epsilon\right).    
\end{equation}

The appropriate one-dimensional limiting problem with the datum $f\in L^2\left(-l, l\right)$ turns out to be 
\begin{equation}\label{2.6}
\left\{
\begin{array}{llll}
\displaystyle
&\left(1-\sigma^2\mathcal{N}\right)\frac{d^2}{dx_1^2}\left(\rho^{n-1}\frac{d^2V}{dx_1^2}\right)+(n-1)\rho^{n-2} V
=(n-1) \rho^{n-2} f, \ \ \ {\rm in}\ \left(-l,l\right),
\vspace{1mm}\\&
V\left(-l\right)=\frac{d V}{dx_1}\left(-l\right)=0,
\vspace{1mm}\\&
V\left(l\right)=\frac{d V}{dx_1}\left(l\right)=0,
\end{array} \right.
\end{equation}
see Lemma~\ref{l 3.1}.

We conclude this section by discussing the coercivity in $H^2(\Omega)$ of the quadratic forms under consideration.

Let $\Omega$ be a bounded domain in $\mathbb{R}^n$ of class $C^{0,1}$. The Sobolev space $H^2\left(\Omega\right)$ is naturally endowed with the norm $\left(\|D^2u\|^2_{L^2\left(\Omega\right)}+\|u\|^2_{L^2\left(\Omega\right)}\right)^{\frac{1}{2}}$.
Note that in \cite{LMC} it was proven that the quadratic form $Q_{\sigma,\Omega}$ is coercive
in $H^2\left(\Omega\right)$ and, moreover, $\left(Q_{\sigma,\Omega}\left(u,u\right)+\|u\|^2_{L^2\left(\Omega\right)}\right)^{\frac{1}{2}}$ is equivalent to the standard norm in $H^2\left(\Omega\right)$. More precisely, if $-1/(n-1)<\sigma <1$, there exists a positive constant $c(n,\sigma)$ depending only on $n$ and $\sigma$ such that  
\begin{equation}
\label{cha}Q_{\sigma,\Omega}\left(u,u\right)\geq c(n,\sigma )\int_{\Omega}|D^2u|^2dx,
\end{equation} 
for all $u\in H^2(\Omega)$.
 In \cite{PDLLP} it is pointed out that $\left(\|D^2u\|^2_{L^2\left(\Omega\right)}+\|u\|^2_{L^2\left(\partial\Omega\right)}\right)^{\frac{1}{2}}$ is equivalent to the standard norm of $H^2\left(\Omega\right)$. Therefore, by \eqref{cha} we can deduce that  $\left(Q_{\sigma,\Omega}\left(u,u\right)+\|u\|^2_{L^2\left(\partial\Omega\right)}\right)^{\frac{1}{2}}$ is also equivalent to the standard norm of $H^2\left(\Omega\right)$.


\section{Proof of Theorem \ref{thm 1.1}}\label{sect. 3}
\noindent
In this section, we prove Theorem ~\ref{thm 1.1}. 

\subsection{Finding the limiting problem}\label{SSec 3.1}
\noindent
We proceed as in \cite{BD} and we write the weak formulation \eqref{2.5} of problem \eqref{2.4} in the form
\begin{equation}\label{3.1}
\begin{aligned}
(1-\sigma)&\int_{\Omega_\epsilon}\left(\frac{\partial^2 u_{\epsilon}}{\partial x^2_1}\frac{\partial^2\phi}{\partial x^2_1}
+2\sum_{i=2}^{n}\frac{\partial^2 u_{\epsilon}}{\partial x_1 \partial x_i}\frac{\partial^2 \phi}{\partial x_1 \partial x_i}
+\sum_{i,j=2}^{n}\frac{\partial^2 u_{\epsilon}}{\partial x_i \partial x_j}\frac{\partial^2\phi}{\partial x_i \partial x_j}\right)dx    
\\&+
\sigma\int_{\Omega_\epsilon}\left(\frac{\partial^2 u_{\epsilon}}{\partial x^2_1}
+\sum_{i=2}^{n}\frac{\partial^2 u_{\epsilon}}{\partial x^2_i}\right)
\left(\frac{\partial^2\phi}{\partial x^2_1}
+\sum_{i=2}^{n}\frac{\partial^2\phi}{\partial x^2_i}\right)dx
\\&+
\mu \int_{\Gamma_{\epsilon}}  
 \left(\frac{\partial u_{\epsilon}}{\partial x_1}\tilde{\nu}_{\epsilon,1}
+\sum_{i=2}^{n}\frac{\partial u_{\epsilon}}{\partial x_i}\tilde{\nu}_{\epsilon,i}\right)
\left(\frac{\partial \phi}{\partial x_1}\tilde{\nu}_{\epsilon,1}
+\sum_{i=2}^{n}\frac{\partial \phi}{\partial x_i}\tilde{\nu}_{\epsilon,i}\right)d\mathcal{H}^{n-1}
\\&+\epsilon
\int_{\Gamma_{\epsilon}}u_{\epsilon}  \phi d\mathcal{H}^{n-1}
=
\int_{\Gamma_\epsilon}f_{\epsilon}  \phi
d\mathcal{H}^{n-1}, \ \text{for all} \ \phi\in H_{L_\epsilon}^2\left(\Omega_\epsilon\right).
\end{aligned}
\end{equation}

We note that the points of  $\Gamma_{\epsilon }$ can be represented as follows
$$(x_1,\epsilon \rho (x_1) x'(\varphi_1,\dots \varphi_{n-3}, \theta))$$
where 
$$x'(\varphi_1,\dots \varphi_{n-3}, \theta)=(x_2(\varphi_1,\dots \varphi_{n-3}, \theta), \dots ,x_n(\varphi_1,\dots \varphi_{n-3}, \theta )$$ 
are written by means of the spherical coordinates of the $n-2$-dimensional unit sphere of ${\mathbb{R}}^{n-1}$ in the form 
$$
\left\{ 
\begin{array}{l}
x_2(\varphi_1,\dots \varphi_{n-3}, \theta)=\cos \varphi_1\\
x_3(\varphi_1,\dots \varphi_{n-3}, \theta)=\sin  \varphi_1\cos \varphi_2\\
x_4(\varphi_1,\dots \varphi_{n-3}, \theta)=\sin  \varphi_1\sin\varphi_2\cos \varphi_3\\
\cdots \cdots \\
\cdots \cdots \\
x_{n-1}(\varphi_1,\dots \varphi_{n-3}, \theta)=\sin  \varphi_1\cdots \sin\varphi_{n-3}\cos \theta\\
x_{n}(\varphi_1,\dots \varphi_{n-3}, \theta)=\sin  \varphi_1\cdots \sin\varphi_{n-3}\sin \theta
\end{array}
\right.
$$
for all $(\varphi_1,\dots ,\varphi_{n-3},\theta)\in [0,\pi)^{n-3}\times [0,2\pi)$. Moreover, one can verify that  the surface measure on $\Gamma_{\epsilon}$ can be written  using these coordinates as 
$$
d\mathcal{H}^{n-1} =\epsilon^{n-2}\rho^{n-2}\sqrt{1+\epsilon^2\rho'^2}dx_1dS (\varphi_1,...,\varphi_{n-3},\theta )$$
where  
$$
dS (\varphi_1,...,\varphi_{n-3},\theta )=\sin\varphi_1^{n-3}\sin\varphi_2^{n-4}\cdots \sin\varphi_{n-3}d\varphi_1\cdots d\varphi_{n-3}d\theta 
$$
is the surface measure of the $n-2$-dimensional unit sphere of ${\mathbb{R}}^{n-1}$.

It is understood that for $n=3$, only the parameter  $\theta $ is involved, whereas for $n=2$ the points of the $\Gamma_{\epsilon}$ are represented by $(x_1,\pm \epsilon \rho (x_1))$ and 
we simply have
$$
d\mathcal{H}^{1}=\sqrt{1+\epsilon^2\rho'^2}dx_1,
$$
the $n-2$-dimensional measure being  the counting measure.  We refer e.g., to \cite[Ex.~65 Ch.~2]{folland}.  

Morever, $\tilde{\nu}_{\epsilon} = (\tilde{\nu}_{\epsilon,1}, \ldots, \tilde{\nu}_{\epsilon,n})$ denotes the unit outer normal  to the boundary $\Gamma_\epsilon$. We note that
\[
\tilde{\nu}_{\epsilon}=\frac{-\epsilon\rho' {\bf{e}}_1+\tilde{\nu}}{\sqrt{1+\epsilon^2\rho'^2}},
\]
where ${\bf{e}}_1=(1,0,\dots , 0)$, $\tilde{\nu}$ is the unit outer normal  to the boundary of the $(n-1)$-dimensional ball 
$$B_{\epsilon}:=B^{n-1}_{\epsilon \rho(x_1)}(x_1)$$
centered at $(x_1,0, \dots 0)$ and with radius $\epsilon \rho(x_1)$. Importantly, $\tilde{\nu}$ is independent of $\epsilon$.

For all $(x_1,x')\in \Omega_1$ we define $v_{\epsilon}(x_1,x')=\epsilon^{\frac{n-2}{2}}u_{\epsilon}(x_1,\epsilon x')$, $g_{\epsilon}(x_1,x')=\epsilon^{\frac{n-2}{2}}f_{\epsilon}(x_1,\epsilon x')$ and $\psi(x_1,x')=\epsilon^{\frac{n-2}{2}}\phi(x_1,\epsilon x')$ in \eqref{3.1}. Then for the function $v_{\epsilon}$ we derive the problem
\begin{equation}\label{3.2}
\begin{aligned}
(1&-\sigma)\int_{\Omega_1}\left(\frac{\partial^2 v_{\epsilon}}{\partial x^2_1}\frac{\partial^2\psi}{\partial x^2_1}
+\frac{2}{\epsilon^2}\sum_{i=2}^{n}\frac{\partial^2 v_{\epsilon}}{\partial x_1 \partial x_i}\frac{\partial^2 \psi}{\partial x_1 \partial x_i}
+\frac{1}{\epsilon^4}\sum_{i,j=2}^{n}\frac{\partial^2 v_{\epsilon}}{\partial x_i \partial x_j}\frac{\partial^2\psi}{\partial x_i \partial x_j}\right)dx    
\\&+
\sigma\int_{\Omega_1}\left(\frac{\partial^2 v_{\epsilon}}{\partial x^2_1}
+\frac{1}{\epsilon^2}\sum_{i=2}^{n}\frac{\partial^2 v_{\epsilon}}{\partial x^2_i}\right)
\left(\frac{\partial^2\psi}{\partial x^2_1}
+\frac{1}{\epsilon^2}\sum_{i=2}^{n}\frac{\partial^2\psi}{\partial x^2_i}\right)dx
\\&+
\frac{\mu}{\epsilon}\int_{\Gamma_1}
 \left(-\epsilon\rho'\frac{\partial v_{\epsilon}}{\partial x_1}
+\frac{1}{\epsilon}\sum_{i=2}^{n}\frac{\partial v_{\epsilon}}{\partial x_i}\tilde{\nu}_i\right)
\left(-\epsilon\rho'\frac{\partial \psi}{\partial x_1}
+\frac{1}{\epsilon}\sum_{i=2}^{n}\frac{\partial \psi}{\partial x_i}\tilde{\nu}_i\right)\frac{d\mathcal{H}_1}{1+\epsilon^2\rho'^2}
\\&+
\int_{\Gamma_1}
v_{\epsilon}\psi
d\mathcal{H}_1
=
\frac{1}{\epsilon}\int_{\Gamma_1}
g_{\epsilon}\psi
d\mathcal{H}_1,
\end{aligned}
\end{equation} 
where we have set 
\begin{equation}\label{3.3}
d\mathcal{H}_1:=
\rho^{n-2}\sqrt{1+\epsilon^2\rho'^2}dx_1dS (\varphi_1,...,\varphi_{n-3},\theta ).    
\end{equation}

We note in passing that by \eqref{cha}, and by changing variables in integrals,  it follows that 
\begin{eqnarray}\label{cha1} \lefteqn{
(1-\sigma)\int_{\Omega_1}\left(\left|\frac{\partial^2 v_{\epsilon}}{\partial x^2_1}\right|^2
+\frac{2}{\epsilon^2}\sum_{i=2}^{n}\left| \frac{\partial^2 v_{\epsilon}}{\partial x_1 \partial x_i}\right|^2
+\frac{1}{\epsilon^4}\sum_{i,j=2}^{n}\left|\frac{\partial^2 v_{\epsilon}}{\partial x_i \partial x_j}\right|^2\right)dx    }
\nonumber \\ & & +
\sigma\int_{\Omega_1}\left|\frac{\partial^2 v_{\epsilon}}{\partial x^2_1}
+\frac{1}{\epsilon^2}\sum_{i=2}^{n}\frac{\partial^2 v_{\epsilon}}{\partial x^2_i}\right|^2
dx\nonumber \\
& & \geq c(n,\sigma ) \int_{\Omega_1}\left(\left|\frac{\partial^2 v_{\epsilon}}{\partial x^2_1}\right|^2
+\frac{2}{\epsilon^2}\sum_{i=2}^{n}\left| \frac{\partial^2 v_{\epsilon}}{\partial x_1 \partial x_i}\right|^2
+\frac{1}{\epsilon^4}\sum_{i,j=2}^{n}\left|\frac{\partial^2 v_{\epsilon}}{\partial x_i \partial x_j}\right|^2\right)dx.
\end{eqnarray}
Then we can prove the following lemma. Note that in this proof and in the rest of the paper,  we use $C$ to denote a positive constant independent of $\epsilon$ that may vary from line to line.
\begin{lemma}\label{l 3.1}
Assume that
\begin{equation}\label{3.4} 
\sup_{\epsilon >0}
\|\epsilon^{-1}g_{\epsilon}\|_{{L^2\left(\Gamma_1\right)}}\ne \infty,
\end{equation}    
and
\begin{equation}\label{3.5}
\epsilon^{-1}g_\epsilon \rightharpoonup g \ \text{in} \ L^2\left(\Gamma_1\right),   
\end{equation}
as $\epsilon\to 0$. 
Then the solution $v_{\epsilon}$ of problem \eqref{3.2} converges to the solution of the problem
\begin{equation}\label{3.6} 
\left\{
\begin{array}{llll}
\displaystyle
&\left(1-\sigma^2\mathcal{N}\right)\frac{d^2}{dx_1^2}\left(\rho^{n-1}\frac{d^2V}{dx_1^2}\right)+(n-1)\rho^{n-2} V
 =(n-1)\rho^{n-2}\mathcal{M}(g), \ \ \ {\rm in }\ \left(-l,l\right),
\vspace{1mm}\\&
V\left(-l\right)=\frac{d V}{dx}\left(-l\right)=0,
\vspace{1mm}\\&
V\left(l\right)=\frac{d V}{dx}\left(l\right)=0,
\end{array} \right.
\end{equation}
as $\epsilon\rightarrow{0}$, in the sense that $v_{\epsilon}\rightharpoonup V$ weakly in  $H^2(\Omega_1)$ and $(v_{\epsilon})_\nu\rightarrow V_\nu$ strongly in  $L^2(\Gamma_1)$. Here
\[
(\mathcal{M}(g))(x_1):=\fint_{\partial B_1}g(x_1,\cdot ) d\mathcal{H}^{n-2}= \frac{1}{(n-1)w_{n-1}\rho^{n-2}}\int_{\partial B_1}g(x_1,\cdot ) d\mathcal{H}^{n-2} ,
\]
 where for $n=2$ it is understood that  
 \[
(\mathcal{M}(g))(x_1)=\frac{1}{2}\left(g(x_1,\rho(x_1))+g(x_1,-\rho(x_1))\right).
\]

In particular, if $g$ depends only on the $x_1$ variable, then the equation in \eqref{3.6} reads
\[\left(1-\sigma^2\mathcal{N}\right)\frac{d^2}{dx_1^2}\left(\rho^{n-1}\frac{d^2V}{dx_1^2}\right)+(n-1)\rho^{n-2} V
 =(n-1)\rho^{n-2}g.
 \]
\end{lemma}

\begin{proof}
First of all, by setting $\psi=v_{\epsilon}$, from \eqref{3.2} we obtain
\begin{equation}\label{3.7}
\begin{aligned}
(1-\sigma)&\int_{\Omega_1}\left|\frac{\partial^2 v_{\epsilon}}{\partial x^2_1}\right|^2
+\frac{2}{\epsilon^2}\sum_{i=2}^{n}\left|\frac{\partial^2 v_{\epsilon}}{\partial x_1 \partial x_i}\right|^2
+\frac{1}{\epsilon^4}\sum_{i,j=2}^{n}\left|\frac{\partial^2 v_{\epsilon}}{\partial x_i \partial x_j}\right|^2dx    
\\&+
\sigma\int_{\Omega_1}\left|\frac{\partial^2 v_{\epsilon}}{\partial x^2_1}
+\frac{1}{\epsilon^2}\sum_{i=2}^{n}\frac{\partial^2 v_{\epsilon}}{\partial x^2_i}\right|^2dx
\\&+
\frac{\mu}{\epsilon}\int_{\Gamma_1}
\left|-\epsilon\rho'\frac{\partial v_{\epsilon}}{\partial x_1}
+\frac{1}{\epsilon}\sum_{i=2}^{n}\frac{\partial v_{\epsilon}}{\partial x_i} \tilde{\nu}_i\right|^2\frac{d\mathcal{H}_1}{1+\epsilon^2\rho'^2}
\\&+
\int_{\Gamma_1}
v^2_{\epsilon} d\mathcal{H}_1=\frac{1}{\epsilon}
\int_{\Gamma_1}
g_{\epsilon}v_{\epsilon}
d\mathcal{H}_1.
\end{aligned}
\end{equation}
Now we estimate the term on the right-hand side of \eqref{3.7}. Namely,
\begin{equation}\label{3.8}
\begin{aligned}
\frac{1}{\epsilon}&
\int_{\Gamma_1}
g_{\epsilon}v_{\epsilon}
d\mathcal{H}_1 
\leq
\frac{1}{2\epsilon^2}
\int_{\Gamma_1}
g^2_{\epsilon}
d\mathcal{H}_1
+
\frac{1}{2}
\int_{\Gamma_1} v^2_{\epsilon}
d\mathcal{H}_1
\\&\qquad\qquad\qquad\qquad \leq
\frac{C}{\epsilon^2}\int_{\Gamma_1}
g^2_{\epsilon}
d\mathcal{H}^{n-1}
+
\frac{1}{2}
\int_{\Gamma_1} v^2_{\epsilon}
d\mathcal{H}_1.
\end{aligned}
\end{equation}
Since $\left(Q_{\sigma,\Omega_1}\left(u,u\right)+\|u\|^2_{L^2\left(\Gamma_1\right)}\right)^{\frac{1}{2}}$ is equivalent to the standard norm of $H^2\left(\Omega_1\right)$ for the functions from $H_{L_1}^2\left(\Omega_1\right)$, combining \eqref{3.4}, \eqref{3.7} and \eqref{3.8}, we deduce that
$\|v_{\epsilon}\|_{H^2(\Omega_1)}\leq C$ for all $\epsilon> 0$, see also \eqref{cha1}; hence, by the compactness of the trace operator, there exists $V\in H^2_{L_1}(\Omega_1)$ such that,
up to a subsequence $v_{\epsilon}\rightharpoonup V$ weakly in  $H^2(\Omega_1)$ and $\nabla v_{\epsilon}\rightarrow \nabla V$ strongly in  $L^2(\Gamma_1)$. We now prove that $V$ is as in the statement.

Note that \eqref{3.4}, \eqref{3.7} and \eqref{3.8} imply that there exist $\tilde{v}_i$ and $\overline{v}_{i,j}$ in $L^2(\Omega_1)$ such that
\begin{equation}\label{3.9}
\frac{1}{\epsilon}\frac{\partial^2 v_{\epsilon}}{\partial x_1 \partial x_i}\rightharpoonup \tilde{v}_i \ \text{in} \ L^2(\Omega_1), \ \text{for all} \ i=2,..,n
\end{equation}
and
\begin{equation}\label{3.10}
\frac{1}{\epsilon^2}\frac{\partial^2 v_{\epsilon}}{\partial x_i \partial x_j}\rightharpoonup \overline{v}_{i,j}\ \text{in} \ L^2(\Omega_1), \ \text{for all} \ i,j=2,..,n.    
\end{equation}
In particular, by setting $v=\sum_{i=2}^n\overline{v}_{i,i}$ we get 
\begin{equation}\label{3.11}
\frac{1}{\epsilon^2}\sum_{i=2}^{n}\frac{\partial^2 v_{\epsilon}}{\partial x^2_i}\rightharpoonup v \ \text{in} \ L^2(\Omega_1).
\end{equation}
The convergences \eqref{3.9}-\eqref{3.11} imply that the limiting function $V$ takes the form 
\begin{equation}\label{3.12}
V(x_1,x')=a_1(x_1)+a_2x_2+\cdot\cdot\cdot +a_n x_n,   
\end{equation}
where $a_1\in H^2\left(-l, l\right)$ and $a_2,...,a_n\in \mathbb{R}$.

We denote by $L^2_{\epsilon, \rho}(\Gamma_1)$ the weighted $L^2$-space 
\[
L^2\biggl(\Gamma_1; \frac{d\mathcal{H}^{n-1} }{\sqrt{1+\epsilon^2\rho'^2}\sqrt{1+\rho'^2}}\biggr),
\]
and note that
\[
\frac{d\mathcal{H}^{n-1}}{\sqrt{1+\epsilon^2\rho'^2}\sqrt{1+\rho'^2}}=\frac{d\mathcal{H}_1}{1+\epsilon^2\rho'^2}.
\]
Then, from  \eqref{3.7} we get 
\[
\frac{\mu}{\epsilon}\int_{\Gamma_1} \left|-\epsilon\rho'\frac{\partial v_{\epsilon}}{\partial x_1} 
+\frac{1}{\epsilon}\sum_{i=2}^{n}\frac{\partial v_{\epsilon}}{\partial x_i} \tilde{\nu}_i\right|^2\frac{d\mathcal{H}_1}{1+\epsilon^2\rho'^2}\leq C,
\]
hence 
\[
\left\|\frac{1}{\epsilon}\sum_{i=2}^{n}\frac{\partial v_{\epsilon}}{\partial x_i}\tilde{\nu}_i-\epsilon\rho'\frac{\partial v_{\epsilon}}{\partial x_1}\right\|_{L^2_{\epsilon, \rho}  (\Gamma_1)}\leq C\epsilon^\frac{1}{2}.
\]
This implies
\begin{equation}\label{3.13}
\begin{aligned}
\left\|\frac{1}{\epsilon}\sum_{i=2}^{n}\frac{\partial v_{\epsilon}}{\partial x_i}\tilde{\nu}_i\right\|_{L^2_{\epsilon,\rho }(\Gamma_1)}
&
\leq C\epsilon^\frac{1}{2}
+\left\|\epsilon\rho'\frac{\partial v_{\epsilon}}{\partial x_1}\right\|_{L^2_{\epsilon ,\rho }(\Gamma_1)}\leq C\epsilon^\frac{1}{2}
+C\epsilon\left\|\text{Tr}\left(\frac{\partial v_{\epsilon}}{\partial x_1}\right)\right\|_{L^2(\Gamma_1)}
\\&
\leq C\epsilon^\frac{1}{2}+C\epsilon\| v_{\epsilon}\|_{H^2(\Omega_1)}\rightarrow0,
\end{aligned}
\end{equation}
as $\epsilon\rightarrow0$, where we have used the boundedness of $\|v_{\epsilon}\|_{H^2(\Omega_1)}$, the boundedness of $\rho'$ and $(\sqrt{1+\epsilon^2\rho'^2})^{-1}$, and the classical Trace Theorem.

Hence, \eqref{3.13}, combined with $\nabla  v_{\epsilon}\rightarrow\nabla V$ strongly in  $L^2(\Gamma_1)$, implies that
\[
\left\|\sum_{i=2}^{n}\frac{\partial V}{\partial x_i}\tilde{\nu}_i\right\|_{L^2(\Gamma_1)}=0.
\]
The last equality combined with \eqref{3.12} gives us $a_2\tilde{\nu}_2+\cdot\cdot\cdot +a_n\tilde{\nu}_n=0$ on $\Gamma_1$. Since $\tilde{\nu}$ is the unit outer normal  to the $(n-1)$-dimensional ball $B^{n-1}_{\rho(x_1)}(x_1)$ and the last equality holds for all values of $\tilde{\nu}$ we get $a_2\equiv\cdot\cdot\cdot \equiv a_n\equiv0$ and $V(x_1,x')=a_1(x_1)$ for all $(x_1,x')\in \Omega_1$. This implies that $V$ is independent on the variables $x'$.

Next we identify the function $v$ in \eqref{3.11}. 
To do so, we select a specific test function $\psi$ in \eqref{3.2}. Let us consider
\[
\psi(x_1,x')=\theta(x_1)\sum_{i=2}^{n}x^2_i,
\]
where $\theta\in C_c^{\infty}\left(-l,l\right)$. Substituting this into \eqref{3.2} and multiplying \eqref{3.2} by $\epsilon^2$, taking the limit as $\epsilon\rightarrow0$, using \eqref{3.9}-\eqref{3.11} and \eqref{3.13} we obtain
\[
(1-\sigma)\int_{\Omega_1}v\theta dx+(n-1)\sigma\int_{\Omega_1}\left(\frac{d^2 V}{dx^2_1}+v\right)\theta dx=0.
\]
Thus, 
\[
\int_{\Omega_1}\left(v+\sigma\mathcal{N}\frac{d^2 V}{dx^2_1}\right)\theta dx=0.
\]
Since the last equality holds for all $\theta\in C_c^{\infty}\left(-l,l\right)$, by the Fubini-Tonelli Theorem we deduce
\[
\int_{B_1}\left(v+\sigma\mathcal{N}\frac{d^2 V}{dx^2_1}\right) dx'=0,
\]
or equivalently
\begin{equation}\label{3.14}
\int_{B_1}v dx'=-w_{n-1}\rho^{n-1}\sigma\mathcal{N}\frac{d^2 V}{dx^2_1}.
\end{equation}

Finally, returning  to problem \eqref{3.2}, choosing a test function $\psi\in  H_0^2\left(-l, l\right)$ (depending only on the variable $x_1$), using the Fubini-Tonelli Theorem and \eqref{3.9}, \eqref{3.11}, \eqref{3.13} and \eqref{3.14}, taking the limit as $\epsilon\rightarrow0$, we obtain
\begin{equation}\label{3.15}
\begin{aligned}
\left(1-\sigma^2\mathcal{N}\right)\int_{-l}^{l}\rho^{n-1} \frac{d^2 V}{dx^2_1}\frac{d^2 \psi}{dx^2_1}dx_1
+(n-1)\int_{-l}^{l}\rho^{n-2} V \psi dx_1
=(n-1)\int_{-l}^{l}\rho^{n-2}\mathcal{M}(g)\psi dx_1.
\end{aligned}
\end{equation}
Since equality \eqref{3.15} holds for all $\psi\in  H_0^2\left(-l, l\right)$, it  represents the weak formulation of the boundary value problem \eqref{3.6}.

\end{proof}

\subsection{Spectral convergence results}
\noindent
In this subsection, we prove the spectral convergence of the eigenvalues and eigenfunctions of problem \eqref{1.2} to the corresponding eigenvalues and eigenfunctions of the one-dimensional problem \eqref{1.4}. 

First, we introduce the Hilbert spaces
\[
\mathcal{H}_\epsilon=L^2(\Gamma_\epsilon; \epsilon^{-2}d\mathcal{H}^{n-1}), \ \text{and} \ \mathcal{H}_0=L^2_{n,\rho}\left(-l, l\right).
\]
Recall that $L^2_{n,\rho}\left(-l, l\right)=L^2\left(\left(-l, l\right); (n-1)w_{n-1}\rho^{n-2}dx_1\right)$.
We define the operator  
\[
\mathcal{E}_\epsilon: \  \mathcal{H}_0\rightarrow  L^2(\Gamma_\epsilon; \epsilon^{-2}d\mathcal{H}^{n-1}),
\]
by setting
\[
\mathcal{E}_\epsilon u=\epsilon^{2-\frac{n}{2}} u, \ \text{for all} \ u\in \mathcal{H}_0.
\]
It is easy to see that $\mathcal{E}_\epsilon$ satisfies condition \eqref{2.1}. 

Now we define the operators $B_\epsilon$  and $B_0$ as required in Sect. \ref{sect. 2}. The operator $B_\epsilon$  will be the resolvent operator associated with problem \eqref{2.5} rescaled by $\epsilon$. Namely, we consider the operator 
\[
B_\epsilon: \ L^2(\Gamma_\epsilon;\epsilon^{-2}d\mathcal{H}^{n-1})\rightarrow  L^2(\Gamma_\epsilon;\epsilon^{-2}d\mathcal{H}^{n-1} ),
\]
defined by
\[
B_\epsilon f_{\epsilon}=\epsilon u_{\epsilon}, \ \text{for all} \ f_{\epsilon}\in L^2(\Gamma_\epsilon;\epsilon^{-2}d\mathcal{H}^{n-1}),
\]
where $u_{\epsilon}$ is the solution to problem \eqref{2.5}. Note that a real number $\lambda(B_\epsilon)\neq0$ is an
eigenvalue of $B_\epsilon$ if and only if $\lambda_\epsilon=\frac{\epsilon(1-\lambda(B_\epsilon))}{\lambda(B_\epsilon)}
$ is an eigenvalue of problem \eqref{1.2} with the same eigenfunction. Equivalently, $\lambda(B_\epsilon) \neq 0$ is an eigenvalue of $B_\epsilon$ if and only if $\overline{\lambda}_\epsilon = \frac{\epsilon}{\lambda(B_\epsilon)}$ is an eigenvalue of problem \eqref{2.3}. Since the trace operator is compact, $B_\epsilon$ is compact.

The operator $B_0$ is the resolvent operator associated with problem \eqref{2.6}. Namely, 
\[
B_0: \ \mathcal{H}_0\rightarrow  \mathcal{H}_0,
\]
such that 
\[
B_{0}f=V, \ \ \text{for all} \ f\in \mathcal{H}_0,
\]
where $V$ is a solution to problem \eqref{2.6}. We note that a real number $\lambda(B_0)\neq0$ is an
eigenvalue of $B_0$ if and only if $\lambda=\frac{1-\lambda(B_0)}{\lambda(B_0)}
$ is an eigenvalue of problem \eqref{1.4} with the same eigenfunction. The operator $B_{0}$
is compact since 
\[
B_{0}\left(\mathcal{H}_0\right)\subset H^2\left(-l, l\right)
\]
and the embedding
\[
H^2\left(-l, l\right)
\hookrightarrow
\mathcal{H}_0
\]
is compact.

\begin{lemma}\label{lemma 3.2}
The following compact convergence holds:
\[
B_\epsilon\xrightarrow{C} B_{0} \ \text{as}  \ \epsilon\rightarrow0,
\]
in the sense of Definition \ref{def 2.4}.
\end{lemma}
\begin{proof}
Let $f_{\epsilon}\in L^2(\Gamma_\epsilon;\epsilon^{-2}d\mathcal{H}^{n-1} )$ be such that
\begin{equation}\label{3.16}
\|f_{\epsilon}\|_{L^2(\Gamma_\epsilon;\epsilon^{-2}d\mathcal{H}^{n-1})}=1.    
\end{equation}
Let $g_{\epsilon}\in L^2(\Gamma_1)$ be defined by $g_{\epsilon}(x_1,x')=\epsilon^{\frac{n-2}{2}}f_{\epsilon}(x_1,\epsilon x')$ for all $(x_1,x')\in \Gamma_1$.  Then by condition \eqref{3.16} we have
\[
\begin{aligned}
\frac{1}{\epsilon^2}&\int_{\Gamma_1 }g^2_\epsilon d\mathcal{H}^{n-1}
\leq
\frac{C}{\epsilon^2}\int_{\Gamma_1} g^2_\epsilon d\mathcal{H}_1
=C\|f_{\epsilon}\|^2_{L^2(\Gamma_\epsilon;\epsilon^{-2}d\mathcal{H}^{n-1})}=C.
\end{aligned}
\]
Hence, assumption \eqref{3.4} holds and, possibly passing to a subsequence, also condition \eqref{3.5} holds for some $g\in L^2(\Gamma_1)$. Then by Lemma \ref{l 3.1}, we conclude that there exists $V \in H^2(\Omega_1)$, independent of the variables $x'$, such that, up to a subsequence $v_{\epsilon} \rightharpoonup V$ in $H^2(\Omega_1)$ and $V$ solves problem \eqref{3.6}. Recall that $B_\epsilon f_{\epsilon} = \epsilon u_{\epsilon}$, where $u_{\epsilon}$ is a solution to problem \eqref{2.5}. Let $v_{\epsilon}(x_1, x') = \epsilon^{\frac{n-2}{2}}u_{\epsilon}(x_1, \epsilon x')$ for all $(x_1, x') \in \Omega_1$. By recalling \eqref{3.3}, we see  that
\[
\|B_\epsilon f_{\epsilon}\|^2_{L^2(\Gamma_\epsilon;\epsilon^{-2}d\mathcal{H}^{n-1} ) )}=\|\epsilon u_{\epsilon}\|^2_{L^2(\Gamma_\epsilon;\epsilon^{-2}d\mathcal{H}^{n-1} ))}
=\int_{\Gamma_1 }v_{\epsilon}^2 d\mathcal{H}_1,    
\]
\[
<\epsilon u_{\epsilon},\mathcal{E}_\epsilon V>_{L^2(\Gamma_\epsilon;\epsilon^{-2}d\mathcal{H}^{n-1} ) }
=\int_{\Gamma_1 }v_{\epsilon} Vd\mathcal{H}_1,
\]
and using the definition of the operator $\mathcal{E}_\epsilon$ and the compactness of trace map, we obtain
\begin{equation}\label{3.17}
\begin{aligned}
\|\epsilon u_{\epsilon}&-\mathcal{E}_\epsilon V\|^2_{L^2(\Gamma_\epsilon;\epsilon^{-2}d\mathcal{H}^{n-1} )}=
\|\epsilon u_{\epsilon}\|^2_{L^2(\Gamma_\epsilon;\epsilon^{-2}d\mathcal{H}^{n-1} )}
\\&
-2<\epsilon u_{\epsilon},\mathcal{E}_\epsilon V>_{L^2(\Gamma_\epsilon;\epsilon^{-2}d\mathcal{H}^{n-1} )}
+\|\mathcal{E}_\epsilon V\|^2_{L^2(\Gamma_\epsilon;\epsilon^{-2}d\mathcal{H}^{n-1} )}\rightarrow 0,
\end{aligned}
\end{equation}
as $\epsilon \rightarrow 0$.

Let $f_{\epsilon}\in L^2(\Gamma_\epsilon;\epsilon^{-2}d\mathcal{H}^{n-1}   )$ and $f\in L^2_{n,\rho}(-l,l)$ be such that
\begin{equation}\label{3.18}
\|f_{\epsilon}-\mathcal{E}_\epsilon f\|_{L^2(\Gamma_\epsilon;\epsilon^{-2}d\mathcal{H}^{n-1} )}\rightarrow 0 \ \text{as} \ \epsilon\rightarrow 0.
\end{equation}
Let $g\in L^2(\Gamma_1)$ be defined by $g=\mathcal{E}_1 f$, and let $g_{\epsilon}\in L^2(\Gamma_1)$ be defined by $g_{\epsilon}(x_1,x')=\epsilon^{\frac{n-2}{2}}f_{\epsilon}(x_1,\epsilon x')$ for all $(x_1,x')\in \Gamma_1$.
Then, using
\begin{multline*}
g((x_1, \rho (x_1) x'(\varphi_1,\dots \varphi_{n-3}, \theta)) )=f(x_1) \\
=\epsilon^{\frac{n-2}{2}}\epsilon^{-1}(\mathcal{E}_\epsilon f)((x_1,\epsilon \rho (x_1) x'(\varphi_1,\dots \varphi_{n-3}, \theta)) ),    
\end{multline*}
and \eqref{3.18} we obtain
\[
\begin{aligned}
\|&\epsilon^{-1}g_{\epsilon}-g\|^2_{L^2\left(\Gamma_1\right)}
 \leq C
\int_{\Gamma_1}\left(\epsilon^{-1}g_\epsilon- g\right)^2
d\mathcal{H}_1
=\frac{C}{\epsilon^2}\int_{\Gamma_{\epsilon} } \left(f_\epsilon-\mathcal{E}_\epsilon f\right)^2 d\mathcal{H}^{n-1}
\rightarrow 0 
\end{aligned}
\]
$\text{as} \ \epsilon\rightarrow 0$.
Then, assumptions \eqref{3.4} and \eqref{3.5} hold and Lemma ~\ref{l 3.1} can be applied. Therefore, we conclude that there exists $V \in H^2(\Omega_1)$ such that, up to a subsequence, $v_{\epsilon} \rightharpoonup V$ in $H^2(\Omega_1)$ and $V$ solves problem \eqref{3.6}. Moreover, by the same computations as in \eqref{3.17}, we conclude that
\[
\|\epsilon u_{\epsilon}-\mathcal{E}_\epsilon V\|_{L^2(\Gamma_\epsilon;\epsilon^{-2}d\mathcal{H}^{n-1})}\rightarrow 0,
\]
as $\epsilon \rightarrow 0$,
thereby completing the proof of Lemma~ \ref{lemma 3.2}.
\end{proof}

\textbf{Proof of Theorem ~\ref{thm 1.1}}. By Lemma \ref{lemma 3.2} and Theorem \ref{th 2.5} it follows that there is spectral convergence of $B_\epsilon$ to $B_0$ as $\epsilon\rightarrow0$.
In particular, for every $k\in \mathbb{N}, \ \lambda_k(B_\epsilon)\rightarrow \lambda_k(B_0)$, as $\epsilon\rightarrow0$. Hence,
\begin{equation}\label{3.19}
\frac{\lambda_{\epsilon,k}}{\epsilon}=\frac{1-\lambda_k(B_\epsilon)}{\lambda_k(B_\epsilon)}\sim\frac{1-\lambda_k(B_0)}{\lambda_k(B_0)}=\lambda_k
\end{equation}
as $\epsilon\rightarrow0$, where $\lambda_{\epsilon,k}$ is the $k$-th eigenvalue of problem \eqref{1.2} and $\lambda_k$ is the $k$-th eigenvalue of problem \eqref{1.4}. The remaining task is to prove the second part of Theorem~\ref{thm 1.1}. 

We begin by observing that if the eigenfunctions $u_{\epsilon ,k}$ are normalized in $L^2(\Gamma_{\epsilon})$ then 
$\tilde u_{\epsilon ,k}:=\epsilon u_{\epsilon ,k}$  are normalized in $L^2(\Gamma_\epsilon;\epsilon^{-2}d\mathcal{H}^{n-1} )$.
Thus,  by Theorem \ref{th 2.5} there exists an orthonormal basis  of eigenfunctions $v_k$, $k\in \mathbb{N}$ of \eqref{1.4} in 
$L^2_{n,\rho}(-l,l)$ such that, possibly passing to a subsequence, 
\begin{equation}\label{vainnikko1}
\|\tilde u_{\epsilon ,k}-\mathcal{E}_\epsilon v_k\|_{L^2(\Gamma_\epsilon;\epsilon^{-2}d\mathcal{H}^{n-1} )}\rightarrow 0 \ \text{as} \ \epsilon\rightarrow 0.
\end{equation} 
In order to prove the convergence in $\Omega_1$, we argue as before. 
For all $(x_1,x') \in \Omega_1$, we set $v_{\epsilon,k}(x_1, x') =\epsilon^{\frac{n-2}{2}} u_{\epsilon,k}(x_1, \epsilon x')$ and $\psi(x_1, x') =\epsilon^{\frac{n-2}{2}} \varphi(x_1, \epsilon x')$ in \eqref{1.3}. Then, for the function $v_{\epsilon,k}$ we obtain the problem
\begin{equation}\label{3.20}
\begin{aligned}
(1&-\sigma)\int_{\Omega_1}\left(\frac{\partial^2 v_{\epsilon,k}}{\partial x^2_1}\frac{\partial^2\psi}{\partial x^2_1}
+\frac{2}{\epsilon^2}\sum_{i=2}^{n}\frac{\partial^2 v_{\epsilon,k}}{\partial x_1 \partial x_i}\frac{\partial^2 \psi}{\partial x_1 \partial x_i}
+\frac{1}{\epsilon^4}\sum_{i,j=2}^{n}\frac{\partial^2 v_{\epsilon,k}}{\partial x_i \partial x_j}\frac{\partial^2\psi}{\partial x_i \partial x_j}\right)dx    
\\&+
\sigma\int_{\Omega_1}\left(\frac{\partial^2 v_{\epsilon,k}}{\partial x^2_1}
+\frac{1}{\epsilon^2}\sum_{i=2}^{n}\frac{\partial^2 v_{\epsilon,k}}{\partial x^2_i}\right)
\left(\frac{\partial^2\psi}{\partial x^2_1}
+\frac{1}{\epsilon^2}\sum_{i=2}^{n}\frac{\partial^2\psi}{\partial x^2_i}\right)dx
\\&+
\frac{\mu}{\epsilon}\int_{\Gamma_1 }
\left(-\epsilon\rho'\frac{\partial v_{\epsilon,k}}{\partial x_1}
+\frac{1}{\epsilon}\sum_{i=2}^{n}\frac{\partial v_{\epsilon,k}}{\partial x_i}\tilde{\nu}_i\right)
\left(-\epsilon\rho'\frac{\partial \psi}{\partial x_1}
+\frac{1}{\epsilon}\sum_{i=2}^{n}\frac{\partial \psi}{\partial x_i}\tilde{\nu}_i\right)\frac{d\mathcal{H}_1}{1+\epsilon^2\rho'^2}
\\=&
\frac{\lambda_{\epsilon,k}}{\epsilon}\int_{\Gamma_1}
v_{\epsilon,k} \psi
d\mathcal{H}_1,
\end{aligned}
\end{equation} 
see also \eqref{3.2}.
If we set in \eqref{3.20} $\psi(x_1, x')=v_{\epsilon,k}(x_1, x')$ for all $(x_1,x') \in \Omega_1$, then taking into account \eqref{3.19} and the normalization $\|u_{\epsilon,k}\|_{L^2(\Gamma_\epsilon)}=1$, we easily obtain that
$\|v_{\epsilon ,k}\|_{H^2(\Omega_1)}\leq C$ for all $\epsilon> 0$; hence, up to a subsequence, $v_{\epsilon ,k} \rightharpoonup V_k$ weakly in $H^2(\Omega_1)$ and $v_{\epsilon ,k} \rightarrow V_k$,  $(v_{\epsilon ,k})_\nu \rightarrow (V_k)_\nu$ strongly in $L^2(\Gamma_1)$. As in Lemma~ \ref{l 3.1} the function $V_k$ is independent of the variables $x'$.

In equation \eqref{3.20}, by selecting a test function $\psi \in H_0^2(-l, l)$ and considering \eqref{3.19}, following the same procedure as in Lemma \ref{l 3.1} and  taking the limit as $\epsilon \rightarrow 0$, we obtain  
\begin{equation}\label{3.21}
\begin{aligned}
\left(1-\sigma^2\mathcal{N}\right)&\int_{-l}^{l}\rho^{n-1} \frac{d^2 V_k}{dx^2_1} \frac{d^2 \psi}{dx^2_1} dx_1
=\lambda_k (n-1)\int_{-l}^{l}\rho^{n-2} V_k \psi dx_1.
\end{aligned}    
\end{equation}
Since equality \eqref{3.21} holds for all $\psi\in  H_0^2\left(-l, l\right)$, then it represents the weak formulation of problem \eqref{1.4} and $V_k$ is an eigenfunction of \eqref{1.4}.

 By re-writing \eqref{vainnikko1} in terms of $u_{\epsilon , k}$, one can easily deduce that, possibly passing to a subsequence,
$v_{\epsilon ,k}\to v_{k}$ in $L^2(\Gamma_1)$ as $\epsilon \to 0$, hence $V_k=v_k$ and the proof is complete. 

\section*{Acknowledgments}
This research has been funded by the Science Committee of the Ministry of Science and Higher Education of the Republic of Kazakhstan (Grant No. AP26194963).

The second named author is a member of the Gruppo Nazionale per l'Analisi  Matematica, la Probabilit\`{a} e le loro Applicazioni (GNAMPA) of the Istituto Nazionale di Alta Matematica (INdAM) and he acknowledges support from the project ``Perturbation problems and asymptotics for elliptic differential equations: variational and potential theoretic methods" funded by the European Union - Next Generation EU and by MUR Progetti di Ricerca di Rilevante Interesse Nazionale (PRIN) Bando 2022 grant 2022SENJZ3.


\begin{thebibliography}{99}

\bibitem{JMA}
Arrieta, J.M., Carvalho, A.N., Losada-Cruz, G.: Dynamics in dumbell domains I. Continuity of the set of equilibria. J. Differ. Equ. \textbf{231}, 551--597  (2006)

\bibitem{PDL}
Arrieta, J.M., Ferraresso, F., Lamberti, P.D.: Spectral analysis of the biharmonic
operator subject to Neumann boundary
conditions on dumbbell domains. Integr. Equ. Oper. Theory \textbf{89}, 377--408 (2017)

\bibitem{arrlam}
Arrieta, J.M., Lamberti, P.D,
Higher order elliptic operators on variable domains. Stability results and boundary oscillations for intermediate
problems,
J. Differential Equations,
\textbf{263}, 4222--4266
(2017)


\bibitem{EA}
Arrieta, J.M., López-Fernández, M., Zuazua, E.: Approximating travelling waves by equilibria
of non-local equations. Asymptot. Anal. \textbf{78}(3), 145--186 (2012)

\bibitem{JM}
Arrieta, J.M., Nakasato, J.M., Pereira, M.C.: The $p$-Laplacian equation in thin domains: the unfolding approach. J. Differ. Equ. \textbf{274}, 1--34 (2021)

\bibitem{VP}
Arrieta, J.M., Villanueva-Pesqueira, M.: Elliptic and parabolic problems in thin domains with doubly weak oscillatory boundary. Commun. Pure Appl. Anal. \textbf{19}(4), 1891--1914 (2020)


\bibitem{FP}
Borisov, D., Freitas, P.: Asymptotics of Dirichlet eigenvalues and eigenfunctions of the Laplacian on thin domains in $\mathbb{R}^d$. J. Funct. Anal. \textbf{258}(3), 893--912 (2010)

\bibitem{BD}
Bucur, D., Henrot A., Michetti, M.: Asymptotic behaviour
of the Steklov spectrum on dumbbell domains. Commun. Partial Differ. Equ. \textbf{46}(2), 362--393 (2021)

\bibitem{buosostek}
Buoso, D.: Analyticity and criticality results for the eigenvalues of the biharmonic operator, in Geometric properties for parabolic and elliptic {PDE}'s, Springer Proc. Math. Stat., \textbf{176}, 65--85, Springer, 2016.

\bibitem{buoken}
Buoso, D., Kennedy, J.B.: The Bilaplacian with Robin boundary conditions. SIAM J. Math. Anal. \textbf{54}(1), 36--78 (2022) 

\bibitem{DBLP}
Buoso, D., Provenzano, L.: A few shape optimization results for a biharmonic Steklov problem. J. Differ. Equ. \textbf{259}(5), 1778--1818 (2015) 

\bibitem{AC}
Carvalho, A., Piskarev, S.: A general approximation scheme for attractors of abstract parabolic problems. Numer. Funct. Anal. Optim. \textbf{27}(7--8), 785--829  (2006)

\bibitem{CD}
Casado-Diaz, J., Luna-Laynez, M., Suarez-Grau, F.J.: A decomposition result for the pressure of a fluid in a thin domain and extensions to elasticity problems. SIAM J. Math. Anal. \textbf{52}(3), 2201--2236 (2020)

\bibitem{LMC}
Chasman, L.M.: An isoperimetric inequality for fundamental tones of free plates with nonzero Poisson’s ratio. Appl. Anal. \textbf{95}(8), 1700--1735  (2016)

\bibitem{ferlamstra}Ferraresso, F., Lamberti P.D. and Stratis I.G., On a Steklov Spectrum in Electromagnetics, in Adventures in Contemporary Electromagnetic Theory,
Edited by Mackay, Tom G. and Lakhtakia, Akhlesh, Springer, Cham, 2023, 195--228.

\bibitem{FEPR}
Ferraresso, F., Provenzano, L.: On the eigenvalues of the biharmonic operator with Neumann boundary conditions on a thin set.
Bull. Lond. Math. Soc. \textbf{55}(3), 1154--1177 (2023)

\bibitem{AFPDL}
Ferrero, A., Lamberti, P.D.: Spectral stability of the Steklov problem. Nonlinear Analysis \textbf{222}(2), 112989  (2022)

\bibitem{SS}
Ferrero, A., Lamberti, P.D.: Spectral stability for a class of fourth order Steklov problems under domain perturbations.  Calc. Var. \textbf{58}(33), 1--57 (2019)

\bibitem{folland} Folland, G.B., Real analysis,
   Pure and Applied Mathematics (New York),
   Modern techniques and their applications,
              A Wiley-Interscience Publication,
 John Wiley \& Sons, Inc., New York, 1984.



\bibitem{GA}
Gaudiello, A., Gomez, D., Perez-Martinez, M.-E.: Asymptotic analysis of the high frequencies for the Laplace operator in a thin T-like shaped structure. J. Math. Pures Appl. \textbf{134}(9), 299--327 (2020)

\bibitem{KJR}
Kuttler, J.R., Sigillito, V.G.: Estimating eigenvalues with a posteriori/a priori inequalities, volume 135
of Research Notes in Mathematics. Pitman (Advanced Publishing Program), Boston, MA, (1985)

\bibitem{PDLLP}
Lamberti, P.D., Provenzano, L.: On the explicit representation of the trace space $H^\frac{3}{2}$ and of the solutions to biharmonic Dirichlet problems on Lipschitz domains via multi-parameter Steklov problems. Rev.
Math. Complut. \textbf{35}, 53--88 (2022)

\bibitem{lampro}
Lamberti, P.D., Provenzano, L.: Viewing the {S}teklov eigenvalues of the {L}aplace operator as
critical {N}eumann eigenvalues,
Current trends in analysis and its applications,
    Trends Math.,
     171--178,
 Birkh\"{a}user/Springer, Cham, 2015.



\bibitem{GL}
Liu, G.: The Weyl-type asymptotic formula for biharmonic Steklov eigenvalues on Riemannian manifolds. Adv. Math. \textbf{228}(4), 2162--2217 (2011)

\bibitem{AP}
Liu, G.: On asymptotic properties of biharmonic Steklov 
eigenvalues. J. Differ. Equ. \textbf{261} 4729--4757  (2016)

\bibitem{JC}
Nakasato, J.C., Pazanin, I., Pereira, M.C.: Reaction-diffusion problem in a thin domain with oscillating boundary and varying order of thickness. Z. Angew. Math. Phys. \textbf{72}(1), Article Number: 5, (2021)

\bibitem{SA}
Nazarov, S.A., Perez, E., Taskinen, J.:  Localization effect for Dirichlet eigenfunctions in thin non-smooth domains. Trans. Amer. Math. Soc. \textbf{368}(7) 4787--4829 (2016)

\bibitem{MC}
Pereira, M.C., Rossi, J.D., Saintier, N.: Fractional problems in thin domains. Nonlinear Anal. \textbf{193}, Article Number: 111471, (2020)

\bibitem{FS}
Stummel, F.: Perturbation of domains in elliptic boundary-value problems. Lecture Notes in Mathematics, Springer, Berlin \textbf{503}, 110--136 (1976)

\bibitem{GMV}
Vainikko, G.M.: Regular convergence of operators and the approximate solution of equations. Math. Anal. \textbf{16}, 5--53 (1979)
 
\end{thebibliography}
\end{document}